\documentclass[AMA,STIX1COL]{WileyNJD-v2}
\usepackage{moreverb}

\newcommand\BibTeX{{\rmfamily B\kern-.05em \textsc{i\kern-.025em b}\kern-.08em
T\kern-.1667em\lower.7ex\hbox{E}\kern-.125emX}}

\articletype{Article Type}%

\received{<day> <Month>, <year>}
\revised{<day> <Month>, <year>}
\accepted{<day> <Month>, <year>}

\usepackage{graphicx}      
\usepackage{multicol}
\usepackage{nicefrac}
\usepackage{graphicx}
\usepackage{epstopdf}
\usepackage{tikz}
\usepackage{color}
\usepackage{bm}
\usepackage{pgfplots}
\usepackage{siunitx}
\usepackage{picture}
\usepackage{epstopdf}
\usepackage{filecontents}
\usepackage{animate}
\usepackage{pgfplotstable}
\usepackage{amsmath}
\usepackage{comment}
\renewcommand{\d}[1]{\ensuremath{\operatorname{d}\!{#1}}}

\newtheorem{assum}{Assumption}
\newtheorem{defn}{Definition}
\usepackage{booktabs}

\begin{document}

\title{Finite-time Newton seeking control}

\author[1]{Martin Guay*}

\author[2]{Mouhacine Benosman}

\authormark{Guay and Benosman}

\address[1]{\orgdiv{Dept. Chemical Engineering}, \orgname{Queen's University}, \orgaddress{\state{Kingston, Ontario}, \country{Canada}}}

\address[2]{\orgname{Mitsubishi Electric Research Laboratories}, \orgaddress{\state{Cambridge, MA}, \country{USA}}}

\corres{*Martin Guay, Department of Chemical Engineering, Queen's University, Kingston, ON, K7L 3N6, Canada. \email{guaym@queensu.ca} }


\abstract[Summary]{This paper proposes a finite-time Newton seeking control design for systems described by  unknown multivariable static maps. The Newton seeking system has an averaged system that implements a Newton continuous-time algorithm. The averaged Newton seeking system is shown to achieve finite-time stability of the unknown optimum of the static map. An averaging analysis demonstrates that the proposed Newton-seeking achieves finite-time practical stability of the optimum. A simulation study is used to demonstrate the effectiveness of the design method.}

%

\maketitle


\section{Introduction}
Finite-time optimization techniques have considerable appeal for the design of real-time optimization for practical applications. The ability to reach the optimum of a cost function of interest in a prescribed time can provide added flexibility in engineering design problems requiring precise synchronous control tasks. In continuous-time systems, gradient-based descent algorithms have been proposed \cite{Cortes2006,ICML20,garg2020}. In \cite{ICML20,garg2020} \footnote{It appears that, albeit using different proof techniques, both groups have come up with similar development in parallel.} a class of discontinuous scaled gradient/Hessian dynamical systems was proposed for the design of continuous-time optimization algorithms with finite-time convergence. These results have been extended in \cite{IJC20} to the case of time-varying cost functions, using a class of scaled second order time-varying continuous optimization flows.

In this study, we propose to exploit second order information in the development of a new class of zero-order Newton finite-time algorithms. While such methods can be readily designed when the cost function {and its derivatives are} known, a limited number of techniques can be used effectively to achieve finite-time stability of the unknown optimum when the cost function is uncertain or unknown. 

In situations where the cost function is unknown but measured, the recommended technique is extremum seeking control (ESC). ESC is a feedback control mechanism designed to drive an unknown nonlinear dynamical system to the optimum of a measured variable of interest \cite{5572972}. ESC is a well established real-time optimization technique with a solid theoretical foundation. Its stability properties, initially characterized in  \cite{KW:SESFGDS} and \cite{Tan2006889}, are well understood. A growing literature has been evolving to complement, generalize and improve the basic schemes. One major extension of ESC is the design of multivariable real-time optimization systems using Newton-based techniques as proposed in \cite{ghaffari2012multivariable}. Newton-seeking techniques implement a Newton update that alleviates performance problems associated with the scaling of the Hessian matrix. An alternative Newton seeking technique was proposed in \cite{labar2019newton} where Lie-bracket averaging techniques were considered. The other main different between the original formulation of \cite{ghaffari2012multivariable} and the reformulation presented in \cite{labar2019newton} is the use of alternative continuous-time Newton optimization techniques. In fact, the approach proposed in \cite{labar2019newton} is an application of  the classical work of Gavurin \cite{gavurin1958nonlinear} as presented in \cite{airapetyan1999continuous}. Such methods avoid the explicit need for the inverse of the Hessian of the unknown cost. It is important to note that continuous-time optimization algorithms have been an active area of research in the optimization literature. Representative recent developments in this area can be found in \cite{massicot2019line}, along with the finite-time generalizations \cite{garg2018new2} and \cite{IJC20}. The main difference with respect to these recent results is the fact that we proceed in an extremum seeking setting (zero-order methods), i.e., we do not assume the knowledge of the gradient or any higher order derivative of the cost function.

 In this manuscript, we propose a Newton seeking technique that can achieve finite-time stability in the practical sense. Given a measured cost function with an unknown mathematical formulation, we propose a multivariable Newton seeking that brings the system to a neighborhood of the unknown optimum value in finite-time. It provides an extension of the gradient based ESC technique proposed in \cite{guayFT2} and \cite{poveda2020fixed}. As in \cite{guayFT2} and in contrast to \cite{poveda2020fixed}, the stability analysis demonstrates that the system yields an averaged dynamical system with a finite-time stable equilibrium at the unknown optimum of the measured cost function. A finite-time Newton seeking technique was originally presented in \cite{guayNFT} without a detailed convergence analysis. In the current paper, we provide a full stability analysis of the proposed finite Newton seeking technique. This includes a detailed analysis of the averaged nonlinear system. This analysis is based on the concept of finite-time input to state stability (FTISS) first introduced by Hong {\it et al} \cite{hong2010FTISS}. As in \cite{guayFT2}, we demonstrate that the averaged system has a finite-time stable equilibrium at the optimum of the unknown cost function. The finite-time stability property of the averaged systems enables the application a classical averaging analysis result \cite{krasnosel1955principle} to show that the closed-loop finite-time Newton seeking system achieves finite-time practical stability of the optimum. 

The paper is structured as follows. Some preliminaries are given in Section \ref{sec1}. The problem formulation is given in Section \ref{sec2}. The proposed ESC is presented in Section \ref{sec3}. Section \ref{sec4} presents a brief simulation study. Conclusions are presented in Section \ref{sec5}.  

%
%
\section{Preliminaries} \label{sec1}
\subsection{Finite-time Stability}

In this section, we present the definition of finite-time stability considered in this study (as stated in \cite{hong2010FTISS}). We introduce the following class of finite-dimensional nonlinear systems:
\begin{align} \label{eq:nonsys}
\dot{x}=F(x)
\end{align}
where $x \in \mathbb{R}^n$ and $F:\mathbb{R}^n\rightarrow \mathbb{R}^n$ is continuous in $x$. 

The continuity of the right hand side of \eqref{eq:nonsys} guarantees existence of at least one solution, possibly non-unique.  The set of all solutions of \eqref{eq:nonsys} with initial conditions $x(t_0)=x_0$  is denoted by $\mathcal{X}(t,t_0,x_0)$ for $t\geq t_0$. In the remainder, the set of all solutions of system \eqref{eq:nonsys} at time $t$ will be simply denoted by $x(t)$. The equilibrium $x_0=0$ is a unique solution of the system in forward time.

\begin{defn} \label{def1}
The equilibrium $x=0$ of {\eqref{eq:nonsys}} is said to be finite-time locally stable if it is Lyapunov stable and such that there exists a settling-time function 
\begin{align*}
T(x_0) = \inf \left \lbrace \bar{T} \geq t_0 \, \left| \, \lim_{t\rightarrow \bar{T}} x(t)=0 \,  ; \, x(t) \equiv 0,\, \forall t\geq \bar{T} \right.\right \rbrace
\end{align*}
in a neighbourhood $U$ of $x=0$. It is globally finite-time stable if $U=\mathbb{R}^n$.
\end{defn}

Finite-time stability can be expressed using a special class of $\mathcal{K}$ functions.  A continuous function $\alpha: \mathbb{R}_{\geq 0} \rightarrow  \mathbb{R}_{\geq 0}$ is a called a class $\mathcal{K}$ function if it is strictly increasing and $\alpha(0)=0$. It is a class $\mathcal{K}_\infty$ function if it is class $\mathcal{K}$ and $\lim_{s \rightarrow \infty} \alpha(s) = \infty$. 

A continuous function $\phi:\mathbb{R}_{\geq 0} \rightarrow  \mathbb{R}_{\geq 0}$ is a generalized class $\mathcal{K}$ function if $\phi(0)=0$ and
\begin{align}
\begin{cases}
\phi(s_1) > \phi(s_2) & \mbox{if}\, \phi(s_1)>0, \, s_1 > s_2 \\
\phi(s_1) = \phi(s_2) & \mbox{if}\, \phi(s_1)=0, \, s_1 > s_2.
\end{cases}
\end{align}
A continuous function $\beta:\mathbb{R}_{\geq 0} \times \mathbb{R}_{\geq 0} \rightarrow  \mathbb{R}_{\geq 0}$ is a generalized $\mathcal{KL}$ function if, for each fixed $t\geq 0$, the function $\beta(s,t)$ is a generalized $\mathcal{K}$ function and each fixed $s\geq 0$, the function $\beta(s,t)$ is such that $\lim_{t \rightarrow T} \beta(s,t)=0$ for $T \leq \infty$.
We can characterize finite-time stability using generalized $\mathcal{K}$ functions as follows: 
\begin{defn} \label{def:FTstab}
System \eqref{eq:nonsys} is finite-time stable if there exists a generalized $\mathcal{KL}$ function $\beta:\mathbb{R}_{\geq 0} \times \mathbb{R}_{\geq 0} \rightarrow  \mathbb{R}_{\geq 0}$  such that every solution $x(t)$ satisfies: $\| x(t) \| \leq \beta(\| x(0) \|,t)$
with $\beta(r,t) \equiv 0$ when $t \geq \bar{T}(r)$ with $\bar{T}(r)$ continuous with respect to $r$ and $\bar{T}(0)=0$. 
\end{defn}
\begin{defn} \label{det:FTLyap}
Let $V(x)$ be a continuous function. It is called a finite-time Lyapunov function if there exists class $\mathcal{K}_\infty$ functions $\phi_1$ and $\phi_2$ and a class $\mathcal{K}$ function $\phi_3$ such that $\phi_1(\| x \|) \leq V(x) \leq \phi_2(\|x\|)$ and
\begin{align*}
D^+ V(x(t)) \leq -\phi_3(\| x\|)
\end{align*}
where, in addition, $\phi_3$ satisfies: $c_1 V(x)^a \leq \phi_3(\|x\|) \leq c_2 V(x)^a$ for some positive constants $a<1$, $c_1>0$ and $c_2>0$. 
\end{defn}
\subsection{Finite-time input-to-state stability (FTISS)}
For systems with inputs, a finite-time version input-to-state stability was proposed in \cite{hong2010FTISS}. It applies to systems of the form:
\begin{align}\label{eq:nonsysv}
\dot{x}=F(x,v(t) )
\end{align}
where $x\in \mathbb{R}^n$. The function $v:\mathbb{R}_{\geq0} \rightarrow \mathbb{R}^m$ is measurable and locally essentially bounded and the vector value function $F:  \mathbb{R}^n\times \mathbb{R}^m \rightarrow \mathbb{R}^n$ is continuous in $x$ and $v(t)$.
\begin{defn} \label{def:FTISS}
System \eqref{eq:nonsysv} is FTISS if there exists a generalized $\mathcal{KL}$ function $\beta:\mathbb{R}_{\geq 0} \times \mathbb{R}_{\geq 0} \rightarrow  \mathbb{R}_{\geq 0}$ and a class $\mathcal{K}$ function $\alpha: \mathbb{R}_{\geq0} \rightarrow \mathbb{R}_{\geq0}$ such that every solution $x(t)$ satisfies:
\begin{align}
\| x(t) \| \leq \beta(\| x(0) \|,t)+\alpha(\| v(t) \|_\infty)
\end{align}
with $\beta(r,t) \equiv 0$ when $t \geq \bar{T}(r)$ with $\bar{T}(r)$ continuous with respect to $r$ and $\bar{T}(0)=0$. 
\end{defn}

Finally, we will need the following definition of practical finite-time stability.
\begin{defn} \label{def:FTstab}
System \eqref{eq:nonsys} is semi-globally practically finite-time stable if there exists a generalized $\mathcal{KL}$ function $\beta:\mathbb{R}_{\geq 0} \times \mathbb{R}_{\geq 0} \rightarrow  \mathbb{R}_{\geq 0}$ and a positive constant $\zeta>0$  such that every solution $x(t)$ starting in $\mathcal{X}$ satisfies:
\begin{align}
\| x(t) \| \leq \beta(\| x(0) \|,t)+ \zeta
\end{align}
with $\beta(r,t) \equiv 0$ when $t \geq \bar{T}(r)$ with $\bar{T}(r)$ continuous with respect to $r$ and $\bar{T}(0)=0$. 
\end{defn}

\section{Problem formulation}  \label{sec2}

We consider a class of multivariable unknown nonlinear systems described by the following dynamical system: 
\begin{subequations}  
	\begin{align}
   \dot{x} &=u \label{eq:sysdyn} \\
   y&=h(x)  \label{eq:sysout}
   \end{align}  
   \end{subequations}
where $x \in \mathbb{R}^p$ are the state variables, $u \in \mathbb{R}^p$ is the input variable, and $y \in \mathbb{R}$ is the output variable.  It is assumed that the function $h:\mathbb{R}^p  \rightarrow \mathbb{R}$ is sufficiently smooth. The function $h$, is assumed to be unknown.  It has an unknown minimizer $x^\ast$ with an optimal value $y^\ast=h(x^\ast)$. 

The cost function, $h(x)$, meets the following assumption. 
\begin{assum}\label{assum:cost1}
The function $h(x)$ has a unique critical point at $x^*$, that is:
$$\left.\frac{\partial h}{\partial x}\right|_{x=x^\ast}=0.$$
The Hessian of $h$ with respect to $x$ at $x^\ast$ is assumed to be positive definite. In particular, there exists a positive constant $\alpha_h$ and a positive definite symmetric matrix $H^\ast$ such that
$$\frac{\partial^2 h (x)}{\partial x \partial x^\top } =H^\ast \geq \alpha_h I$$ for all $x \in \mathcal{X} \subset \mathbb{R}$. \end{assum}
  
The objective of this study is to develop an ESC design technique that guarantees finite-time convergence to a neighbourhood of the unknown minimizer, $x^\ast$, of the measured function $y=h(x)$. 

 \section{Finite-time Newton seeking controller design and analysis} \label{sec3}
 
\subsection{Proposed target average Newton seeking system}

In this study, we seek a Newton seeking system whose trajectories, subject to choice of suitable dither signals, follow those of that can be achieved using a judicious choice of dither signals. We consider the following target averaged system:
\begin{equation} \label{eq:tarsys}
\begin{aligned} 
\dot{x} =& k \gamma(v) v \\
\dot{v} =&-K \gamma(H v + g)(H v + g)  
\end{aligned}
\end{equation}
where $k$, $K$ are controller gains to be assigned. The matrix $H=\nabla_x^2 h$ is the Hessian of the objective function. The vector $g=\nabla_x h$ denotes the gradient of the cost function.

The function $\gamma(v)$ is given by (\cite{P12,REPP18}):
\begin{align*}
\gamma(v) = \dfrac{c_1}{\| v \|^{\alpha_{1}}} + \dfrac{c_2}{\| v \|^{\alpha_{2}}}
\end{align*}
where $\alpha_1=\dfrac{q_1-2}{q_1-1}$ and $\alpha_2=\dfrac{q_2-2}{q_2-1}$ for $q_1 \in (2,\infty)$ and $q_2 \in (1,2)$.

The function $\gamma(v) v$ is locally Lipschitz continuous everywhere except at $v=0$. It is continuous everywhere. The function $\gamma(H v + g)(H v + g)$ is also not locally Lipschitz continuous for $\forall x$, $\forall v$ that yield a gradient $g$ and Hessian $H$ that satisfy the equation $v=-H^{-1}g$. It is only continuous on this set but locally Lipschitz everywhere else. 

\begin{remark}
We note that the proposed target average system provides a finite-time generalization of the classical continuous-time Newton algorithms proposed in \cite{gavurin1958nonlinear} and \cite{airapetyan1999continuous}. These classical techniques were also considered for the design of Newton seeking system in \cite{labar2019newton}. The main difference with respect to these classical results resides in the finite-time convergence and the possibility to directly tune a pre-defined or prescribed finite convergence time as function of the flow gains $k,\;c_{1},\;c_{2}$, e.g., refer to \cite{ICML20,garg2020,IJC20}.
\end{remark}

We define the coordinate transformation $z=Hv+g$ and rewrite the dynamics as:
\begin{equation}\label{eq:tarsys2}
\begin{aligned}
\dfrac{\d z}{\d t} =& \, -K \gamma(z) H  z +k \gamma(H^{-1}(z-g))(z-g)\\
\dfrac{\d g }{\d t}=&\, k \gamma(H^{-1}(z-g))(z-g) 
\end{aligned}
\end{equation}

This system is such that it has a unique equilibrium at the point $z=0$ and $g=0 \Rightarrow x=x^\ast$. It is continuous everywhere and locally Lipschitz continuous away from the equilibrium and the plane  $z-g=0$. 

The following Lemma establishes the finite-time stability analysis of the equilibrium of \eqref{eq:tarsys2}. 

\begin{lemma}\label{lem1}
Consider the nonlinear system \eqref{eq:tarsys}. Let Assumption \ref{assum:cost1} be satisfied. Then for any $k$, there exists a $K^\ast$ such that for any $K>K^\ast>0$, the optimum $x=x^\ast$, $v=0$, is a globally finite-time stable equilibrium of  \eqref{eq:tarsys}. 
\end{lemma}
\begin{proof}
We pose the candidate Lyapunov function $V_1=\dfrac{1}{2} z^\top z+ \dfrac{1}{2} g^\top g$. Its derivative along the trajectories of \eqref{eq:tarsys2} yields:
\begin{align*}
\dot{V}_1 =-K \gamma(z) z^\top Hz +k\gamma(H^{-1}(z-g)) (z+g)^\top (z-g) 
\end{align*}
This can be rewritten in the following manner:
\begin{align*}
\dot{V}_1 =& -K \gamma(z) z^\top Hz  - k\gamma(H^{-1}(z-g)) (z-g)^\top (z-g) +2k \gamma(H^{-1}(z-g)) z^\top (z-g)
\end{align*}
Let us first consider the function $\phi(z, g)= \gamma(H^{-1}(z-g)) (z-g)^\top (z-g)$. It is given explicitly as:
\begin{align}
\phi(z, g) =& \dfrac{c_1 (z-g)^\top (z-g)}{\|H^{-1}(z-g)\|^{\alpha_{1}}} + \dfrac{c_2(z-g)^\top (z-g)}{\|H^{-1}(z-g)\|^{\alpha_{2}}}.
\end{align}
For $0<\alpha_1<1$, the first term can be lower bounded as follows:
\begin{align*}
\dfrac{c_1 (z-g)^\top (z-g)}{\|H^{-1}(z-g)\|^{\alpha_{1}}} &\geq \dfrac{c_1 (z-g)^\top (z-g)}{\|H^{-1}\|^{\alpha_{1}} \|(z-g)\|^{\alpha_{1}}}\geq \dfrac{c_1 (z-g)^\top (z-g)}{\lambda_{max}(H^{-1})^{\alpha_{1}} \|(z-g)\|^{\alpha_{1}}}\\
& = \dfrac{c_1\lambda_{min}(H)^{\alpha_1}  (z-g)^\top (z-g)}{ \|(z-g)\|^{\alpha_{1}}} =c_1 \lambda_{min}(H)^{\alpha_1} \| z-g\|^{2-\alpha_1}.
\end{align*}
Similarly for $\alpha_2<0$, we obtain:
\begin{align*}
c_2 \|(z-g)\|^2 &\|H^{-1}(z-g)\|^{-\alpha_{2}} \geq {c_2} \lambda_{min}(H^{-1})^{-\alpha_{2}} \|(z-g)\|^{2-\alpha_{2}}=c_2 \lambda_{max}(H)^{\alpha_2} \| z-g\|^{2-\alpha_2}.
\end{align*}

Next, we consider the function $\beta(z,g)=\gamma(H^{-1}(z+g)) z^\top (z-g)$:
\begin{align} \label{eq:ineq3}
\beta(z, g) =& \dfrac{c_1 z^\top (z-g)}{\|H^{-1}(z-g)\|^{\alpha_{1}}} +c_2 \dfrac{z^\top (z-g)}{\|H^{-1}(z-g)\|^{\alpha_{2}}}.
\end{align}

The first term of \eqref{eq:ineq3} can be bounded as follows:
\begin{align*}
 \dfrac{c_1 z^\top (z-g)}{\|H^{-1}(z-g)\|^{\alpha_{1}}} &\leq \dfrac{c_1 z^\top H H^{-1} (z-g)}{\|H^{-1}(z-g)\|^{\alpha_{1}}} \leq c_1 \|H z\| \|H^{-1}(z-g)\|^{1-\alpha_1}.
\end{align*}
Using Young's inequality, one can write:
\begin{align*}
 \dfrac{c_1 z^\top (z-g)}{\|H^{-1}(z-g)\|^{\alpha_{1}}} \leq & c_1 \dfrac{k_1}{2-\alpha_1} \| H z\|^{2-\alpha_1}+ \dfrac{c_1(1-\alpha_1)}{2-\alpha_1} k_1^{-1/(1-\alpha_{1})} \| H^{-1} (z-g)\|^{2-\alpha_1} \\
 &\leq c_1 \dfrac{k_1}{2-\alpha_1}  \lambda_{max}(H)^{2-\alpha_1} \|z\|^{2-\alpha_1} + \dfrac{c_1(1-\alpha_1)}{2-\alpha_1} k_1^{-1/(1-\alpha_{1})} \lambda_{min}(H)^{\alpha_1-2} \|z-g\|^{2-\alpha_1} 
 \end{align*}
 for some positive constant $k_1>0$. 
 Similarly, one can bound the second term of \eqref{eq:ineq3} to obtain:
 \begin{align*}
 \dfrac{c_2 z^\top (z-g)}{\|H^{-1}(z-g)\|^{\alpha_{2}}} \leq &c_2\|H z\| \|H^{-1}(z-g)\|^{2-\alpha_{2}}\\
 & \leq c_2 \dfrac{k_2}{2-\alpha_2}  \lambda_{max}(H)^{2-\alpha_2} \|z\|^{2-\alpha_2} + \dfrac{c_2(1-\alpha_2)}{2-\alpha_2} k_2^{-1/(1-\alpha_{2})} \lambda_{min}(H)^{\alpha_2-2} \|z-g\|^{2-\alpha_2}  \end{align*}
 for some positive constant $k_2>0$. We can then write $\dot{V}_1$ as:
\begin{align*}
\dot{V}_1 \leq &-K c_1\lambda_{min}(H)\|z\|^{2-\alpha_1} -K c_2 \lambda_{min}(H) \| z\|^{2-\alpha_2}  -k c_1 \lambda_{min}(H)^{\alpha_{1}}  \| z-g\|^{2-\alpha_1}- k c_2 \lambda_{max}(H)^{\alpha_{2}} \| z-g\|^{2-\alpha_2} \\
&+2k c_1 \dfrac{k_1}{2-\alpha_1}  \lambda_{max}(H)^{2-\alpha_1} \|z\|^{2-\alpha_1} +2 \dfrac{k c_1(1-\alpha_1)}{2-\alpha_1} k_1^{-1/(1-\alpha_{1})} \lambda_{min}(H)^{\alpha_1-2} \| z-g\|^{2-\alpha_1} \\
 &+2k c_2 \dfrac{k_2}{2-\alpha_2}  \lambda_{max}(H)^{2-\alpha_2} \|z\|^{2-\alpha_2} + 2\dfrac{k c_2(1-\alpha_2)}{2-\alpha_2} k_2^{-1/(1-\alpha_{2})} \lambda_{min}(H)^{\alpha_2-2} \| z-g\|^{2-\alpha_2}
 \end{align*}
 
It is then straightforward to show that for any:
$$k_1>\left(\dfrac{1}{\lambda_{min}(H)^2}\dfrac{1-\alpha_1}{2-\alpha_1}\right)^{1-\alpha_1}, \,\,\,k_2>\left(\dfrac{\lambda_{max}(H)^{\alpha_2-2}}{\lambda_{min}(H)^{\alpha_{2}}}\dfrac{1-\alpha_2}{2-\alpha_2}\right)^{1-\alpha_2}$$
and 
$$K>2 k\, \mbox{max}\left[ \frac{ k_1}{2-\alpha_1} \frac{\lambda_{max}(H)^{2-\alpha_1}}{\lambda_{min}(H)}, \,  \frac{ k_2}{2-\alpha_2} \frac{\lambda_{max}(H)^{2-\alpha_2}}{\lambda_{min}(H)} \right],$$
there exist positive nonzero constants $k_3$, $k_4$, $k_5$ and $k_6$ such that:
\begin{align*}
\dot{V}_1 \leq &-k_3 \|z\|^{2-\alpha_1} -k_5 \| z\|^{2-\alpha_2}  -k_4\| z-g\|^{2-\alpha_1}- k_6  \| z-g\|^{2-\alpha_2}.
\end{align*}
As a result, we obtain the following inequality:
\begin{align*}
\dot{V}_1 \leq &-k_3 \|z\|^{2-\alpha_1} -k_5\| z-g\|^{2-\alpha_1}.
\end{align*}

From the derivation, one can pick $K$ large enough such that $k_3> k_5$. Consequently, we can write:
\begin{align*}
\dot{V}_1 \leq &-(k_3-k_5) \|z\|^{2-\alpha_1} -k_5(\| z-g\|^{2-\alpha_1}+\|z\|^{2-\alpha_1}).
\end{align*}
Since, $\|g\|^{2-\alpha_1}\leq \|z\|^{2-\alpha_1} + \| z-g\|^{2-\alpha_1}$, we can deduce that:
\begin{align*}
\dot{V}_1 \leq &-(k_3-k_5) \|z\|^{2-\alpha_1} -k_5\|g\|^{2-\alpha_1}.
\end{align*}
If we let $k_7 = \min[(k_3-k_5),k_5]>0$, it follows that:
\begin{align*}
\dot{V}_1 \leq &-k_7 \|z\|^{2-\alpha_1}+k_7 \|g\|^{2-\alpha_1}
\end{align*}
or,
\begin{align*}
\dot{V}_1 \leq &-k_7 (2 V_1)^{(2-\alpha_1)/2}=-2^{(2-\alpha_1)/2} k_7 V_1^{(2-\alpha_1)/2}.
\end{align*}

Since $(2-\alpha_1)/2<1$, it follows that the proposed target system has a globally finite-time equilibrium at $z=0$ and $g=0$. As a result, the point $x=x^\ast$, $v=0$, is a globally finite-time stable equilibrium of \eqref{eq:tarsys}. This completes the proof.
 \end{proof}

\subsection{Finite-time Newton seeking system}\label{sec4}
We first define the following signals:
\begin{align*}
M(t)=[\sin(\omega_1 t), \, \ldots, \, \sin(\omega_n t) ]^\top ,
\end{align*}
We also define the following signals: $S_{ii}(t)=\dfrac{16}{a^2}\bigg( \sin(\omega_i t)^2 -\dfrac{1}{2} \bigg)$
for $i=1,\,\ldots, n$ and $S_{ij}(t)=\dfrac{8}{a} \sin(\omega_i t)\sin(\omega_j t)$ 
for $i\neq j$, $i,j=1, \, \ldots, \,n$. Each frequency are chosen such that the sinusoidal dither signals are mutually orthogonal.

We let $S(t)$ be the $n \times n$ matrix with elements $S_{ij}$. We define the following vector $s(t)= vec(S(t))$.
The operator $vec: \mathbb{R}^{n \times n} \rightarrow \mathbb{R}^{\frac{n(n+1)}{2}}$ vectorizes the symmetric matrix $S(t)$ as follows:
\begin{align*}
s(t)=\left[ S_{11} ,\, S_{12},\, \ldots \, S_{1n}, \ S_{22}, \, \ldots, \, S_{2n} \, \ldots, \, S_{nn} \right]^\top .
\end{align*}

The proposed closed-loop finite-time Newton seeking system is given:

\begin{equation} \label{eq:ftNS}
\begin{aligned}
\dfrac{\d x}{\d t} =& \,{k}\gamma(v) v \\
\dfrac{\d v }{\d t}=&\, -\gamma(Hv+\delta_1) K (Hv +\delta_1)\\
\dfrac{\d\xi }{\d t}=&\, -K_2 \gamma(\xi-\delta_2)(\xi-\delta_2) 
\end{aligned}
\end{equation}
where  $K_2>0$ is a positive constant, $\xi=vec(H)$,  $\delta_1(x,t)=\dfrac{2}{a} h(x+a M(t))M(t)$  and $\delta_2(x,t)=h(x+a M(t))s(t).$ 


A formal average of this system is given by:
\begin{equation} \label{eq:crazav}
\begin{aligned}
\dfrac{\d x^{a}}{\d t} =& \,{ k}\gamma(v^{a}) v^{a} \\
\dfrac{\d v^{a}}{\d t}=&\,  -\frac{K}{T} \int_0^T \bigg( \frac{c_1 (H^{a} v^{a} + {\delta_1}(x^{a},t))}{\|H^{a} v^{a} + {\delta_1}(x^{a},t))\|^{\alpha_{1}}} + \frac{c_2 (H^{a} v^{a} +{ \delta_1}(x^{a},t))}{\|H^{a} v^{a} +{\delta_1}(x^{a},t))\|^{\alpha_{2}}} \bigg) dt \\
\dfrac{\d\xi^{a} }{\d t}=& \, -\frac{K_2}{T} \int_0^T \bigg( \frac{c_1 (\xi^{a} - \delta_2(x^{a},t))}{\|\xi^{a} - \delta_2(x^{a},t))\|^{\alpha_{1}}} + \frac{c_2 (\xi^{a}- \delta_2(x^{a},t))}{\|\xi^{a} - \delta_2(x^{a},t))\|^{\alpha_{2}}} \bigg) dt.
\end{aligned}
\end{equation}

The term $Hv + \delta_1(x,t)$ can be expanded in the following manner:
\begin{align*}
Hv + \delta_1(x,t) =& Hv+ \frac{2}{a} h(x + a { M(t)}) { M(t)}\\
=&  \frac{1}{a} \bigg( 2 h(x){M(t)}+2\nabla h(x) { M(t)^\top M(t)}a +Hva + a^2M(t)^\top  \nabla^2 h(x) { M(t).} \bigg).
\end{align*}
%
%
As a result, we obtain:
\begin{align*}
\frac{1}{T} \int_{0}^T (Hv+\delta_1(x,t) ) dt \approx& (Hv+ \nabla_x h).
\end{align*}

For the term $\xi-\delta_2(t,x)$, one can show that:
\begin{align*}
\frac{1}{T} \int_{0}^T (\xi-\delta_2(x,t) ) dt \approx& (\xi - vec(H^\ast)),
\end{align*}
%
{where $H^\ast$ is the Hessian of $h(x)$}. The next result considers the simplified case in which the Hessian estimate $\xi^{a}$ is at its equilibrium $\xi^{a} = vec(H^\ast)$.
\begin{lemma} \label{lem2}
Consider the nonlinear system \eqref{eq:crazav}. Let Assumption \ref{assum:cost1} be satisfied and further assume that $\xi^{a}=vec(H^\ast)$. Then the optimum $x^{a}=x^\ast$, $v^{a}=0$ is a finite-time stable equilibrium of the system. 
\end{lemma}
\begin{proof} The proof is given in Appendix. \end{proof}

Next, we consider the stability of the averaged system \eqref{eq:crazav}. This system is viewed as the interconnection of two systems. The first system is given by the dynamics of {$x^{a}$ and $v^{a}$} while the second system consists of the Hessian update dynamics $\xi^{a}$. 

\begin{lemma}\label{lem3}
Consider the nonlinear system \eqref{eq:crazav}. Let Assumption \ref{assum:cost1} be satisfied and further assume that the Hessian update dynamics are such that the Hessian matrix $H^{a}$ corresponding to $\xi^{a}=vec(H^{a})$ is positive definite $\forall t \geq0$. Then there exists a $K^\ast$ such that, for any $K>K^\ast>0$, the optimum $x^{a}=x^\ast$, $v^{a}=0$, $\xi^{a}=vec(H^\ast)$ is a finite-time stable equilibrium of the system. 
\end{lemma}
\begin{proof}
We first consider the averaged Hessian update dynamics and define the Hessian estimation error $\tilde{\xi}^{a} = \xi^{a}-vec(H^\ast)$. The error dynamics take the form:
{ \begin{align*}
\dot{\tilde{\xi}}^{a}=&-\frac{K_2}{T}  \int_0^T \bigg(  \frac{c_1 (\xi^{a} - \delta_2(x^{a},t))}{\|\xi^{a} - \delta_2(x^{a},t))\|^{\alpha_{1}}}+ \frac{c_2 (\xi^{a} - \delta_2(x^{a},t))}{\|\xi^{a} - \delta_2(x^{a},t))\|^{\alpha_{2}}} \bigg) dt. 
\end{align*}}
 We pose the Lyapunov function candidate: $V_2^{a}=\frac{1}{2} \| \tilde{\xi}^{a} \|^2$. Its derivative is given by:
\begin{align*}
\dot{V}_2^{a} =& -\frac{K_2}{T} \tilde{\xi}^{a^{T}} \int_0^T \bigg(  \frac{c_1 (\xi^{a} - \delta_2(x^{a},t))}{\|\xi^{a} - \delta_2(x^{a},t))\|^{\alpha_{1}}}+ \frac{c_2 (\xi^{a} - \delta_2(x^{a},t))}{\|\xi^{a} - \delta_2(x^{a},t))\|^{\alpha_{2}}} \bigg) dt.
\end{align*}
It is easy to show that: $\tilde{\xi}^{a} = \frac{1}{T} \int_0^T (\xi^{a} - \delta_2(x^{a},t)))dt.$ As a result, we can proceed as in Lemma \ref{lem2} and show that there exists a positive constant $k_9>0$ such that:
\begin{align*}
\dot{V_2}^{a} \leq -2^{(2-\alpha_1)/2} k_9 V_2^{a^{(2-\alpha_1)/2}}.
\end{align*}
Therefore, the averaged Hessian estimation update has a globally finite-time stable equilibrium at $\xi^{a}=vec(H^\ast)$.  

Now we consider the Newton seeking algorithm {\eqref{eq:crazav}} subject to the coordinate transformation $z^{a}=H^{a} v^{a}+\nabla h(x^{a})$ and $g^{a}=\nabla h(x^{a})$. This yields the averaged dynamics: 
\begin{equation}\label{eq:crazav3}
\begin{aligned}
\dfrac{\d g^{a}}{\d t} =& \,{ k}H^{\ast} \gamma(v^{a}) (v^{a}) \\
\dfrac{\d z^{a}}{\d t}=&\,  -\frac{K}{T}{H^a} \int_t^{t+T} \bigg( \frac{c_1 (H^{a} v^{a}+ \delta_1(x^{a},t))}{\| H^{a} v^{a}+ \delta_1(x^{a},t)\|^{\alpha_{1}}} + \frac{c_2 (H^{a} v^{a}+ \delta_1(x^{a},t))}{\|H^{a} v^{a}+ \delta_1(x^{a},t)\|^{\alpha_{2}}} \bigg) dt   +{k}\gamma(v^{a})H^\ast  v^{a}+ \dot{H}^{a} v^{a}.
\end{aligned}
\end{equation}
Under the simplifying assumption that $H^{a}$ is nonsingular for all $t\geq 0$, one can write $v^{a}={H^{a}}^{-1}(z^{a}-g^{a})$.  We pose the Lyapunov function candidate $V_3^{a}=\frac{1}{2} \| z^{a}\|^2 + \frac{1}{2} \|g^{a}\|^2$. Its derivative is given by:
\begin{align*}
\dot{V}_3^{a} =& -\frac{K}{T} {z^{a}}^\top H^{a} \int_t^{t+T} \bigg(  \frac{c_1 (H^{a} v^{a} + \delta_1(x^{a},t))}{\|H^{a}v^{a} + \delta_1(x^{a},t))\|^{\alpha_{1}}}+ \frac{c_2 (H^{a} v^{a} + \delta_1(x^{a},t))}{\|H^{a} v^{a} + \delta_1(x^{a},t))\|^{\alpha_{2}}} \bigg) dt \\
&+ k {(z^{a}+g^{a})}^\top {H^a} \gamma(v^{a}) v^{a}+ {z^{a}}^{\top} \dot{H}^{a} v^{a} 
\end{align*}
The dynamics of the Hessian update are written as:
\begin{align*}
\dot{H}^{a} = -&K_2 \frac{1}{T} \int_0^T \frac{c_1 (H^{a}-h(x+aM(t))S(t))}{\| \xi^{a} - \delta_2(x^{a},t) \|^{\alpha_{1}}} dt -K_2 \frac{1}{T} \int_0^T \frac{ c_2(H^{a}-h(x+aM(t))S(t))}{\| \xi^{a} - \delta_2(x^{a},t) \|^{\alpha_{2}}} dt 
\end{align*}
If one considers the first term, we can provide the following bound:
\begin{align*}
 c_1 K_2 &\frac{1}{T} \int_0^T \frac{ \|\xi^{a}-\delta_2(x^{a},t)\| }{\| \xi^{a} - \delta_2(x^{a},t) \|^{\alpha_{1}}} dt 
\leq \frac{c_1 K_2}{T^{\alpha_1}} \bigg(\frac{1}{T}  \int_0^T \|\xi^{a}-\delta_2(x^{a},t)\|dt \bigg)^{1-\alpha_1}= \frac{c_1 K_2}{T^{\alpha_1}}\|\tilde{\xi}^{a}\|^{1-\alpha_1}
\end{align*}
Using the same approach for the second term, one obtains:
\begin{align*}
\|\dot{H}^{a}\| \leq \frac{K_2 c_1}{T^{\alpha_1}} \|\tilde{\xi}^{a}\|^{1-\alpha_1}+ \frac{K_2 c_2}{T^{\alpha_2}} \|\tilde{\xi}^{a}\|^{1-\alpha_2} =\phi(\|\tilde{\xi}^{a}\|)
\end{align*}
where $\phi$ is a class $\mathcal{K}^\infty$ function. 

%
We can then write the following upper bound for $\dot{V}_3^{a}$, 
\begin{align*}
\dot{V}_3^{a} \leq& -\frac{K}{T} {z^{a}}^\top  H^{a} \int_t^{t+T} \bigg(  \frac{c_1 (H^{a} v^{a} + \delta_1(x^{a},t))}{\|H^{a}v^{a} + \delta_1(x^{a},t))\|^{\alpha_{1}}}+ \frac{c_2 (H^{a} v^{a} + \delta_1(x^{a},t))}{\|H^{a} v^{a} + \delta_1(x^{a},t))\|^{\alpha_{2}}} \bigg) dt \\
&+ k {z^{a}}^\top {H^\ast} \gamma(v^{a}) v^{a} + k g^{a^{T}} {H^\ast} \gamma(v^{a})v^{a} + \|z^{a} \| \|v ^{a} \| \phi(\|\tilde{\xi}^{a} \|).
\end{align*}
Next, we define $\tilde{H}^a=H^\ast-H^a$ and rewrite the last inequality
\begin{align*}
\dot{V}_3^{a} \leq& -\frac{K}{T} (z^{a})^\top  H^{a} \int_t^{t+T} \bigg(  \frac{c_1 (H^{a} v^{a} + \delta_1(x^{a},t))}{\|H^{a}v^{a} + \delta_1(x^{a},t))\|^{\alpha_{1}}}+ \frac{c_2 (H^{a} v^{a} + \delta_1(x^{a},t))}{\|H^{a} v^{a} + \delta_1(x^{a},t))\|^{\alpha_{2}}} \bigg) dt \\
&+ k {z^a}^\top \tilde{H}^a \gamma(v^a)v^a+k {g^a}^\top \tilde{H}^a \gamma(v^a)v^a  + k {z^{a}}^\top {H^a} \gamma(v^{a}) v^{a} + k {g^{a}}^\top {H^a} \gamma(v^{a})v^{a} \\&+ \|z^{a} \| \|v ^{a} \| \phi(\|\tilde{\xi}^{a} \|).
\end{align*}

Proceeding as in Lemma \ref{lem2}, we obtain:
\begin{align*}
\dot{V}_3^{a} \leq& -K \lambda_{min}(H^{a}) \left(c_1\frac{\|z^{a}\|^{2-\alpha_1}}{T} + c_2\frac{\|z^{a}\|^{2-\alpha_2}}{T^{\alpha_2}}\right)+ k {z^a}^\top \tilde{H}^a \gamma(v^a)v^a+k {g^a}^\top \tilde{H}^a \gamma(v^a)v^a  \\
&+ k {z^{a}}^\top {H^a} \gamma(v^{a}) v^{a} + k {g^{a}}^\top {H^a} \gamma(v^{a})v^{a}+ \|z^{a} \| \|v ^{a} \| \phi(\|\tilde{\xi}^{a} \|).
\end{align*}

%
Upon substitution of $v^{a}=H^{a^{-1}}(z^{a}-g^{a})$, we obtain:
\begin{align*}
\dot{V}_3^{a} &\leq -K \lambda_{min}(H^{a}) \left(c_1 \frac{\|z^{a}\|^2}{T\|z^{a}\|^{\alpha_1}} + c_2 \frac{\|z^{a}\|^2}{T^{\alpha_2}\|z^{a}\|^{\alpha_2}}\right)+2 k {z^{a}}^\top \gamma(H^{a^{-1}}(z^{a}-g^{a})) (z^{a}-g^{a})\\
&- k (z^{a}-g^{a})^{\top} \gamma(H^{a^{-1}}(z^{a}-g^{a})) (z^{a}-g^{a})+ k (z^{a}+g^a)^\top \tilde{H}^{a} \gamma(H^{a^{-1}}(z^{a}-g^{a})) H^{a^{-1}}(z^{a}-g^{a}) \\
&+ \|z^{a} \| \|H^{a^{-1}}(z^{a}-g^{a}) \| \phi(\|\tilde{\xi}^{a} \|).
\end{align*}
%
%
Next, we proceed as in Lemma \ref{lem1} to bound all indeterminate terms in the last inequality. This gives the following inequality:
\begingroup
\allowdisplaybreaks
\begin{align*}
&\dot{V}_3^{a} \leq -\frac{K}{T} c_1\lambda_{min}(H^{a})\|z^{a}\|^{2-\alpha_1} -\frac{K}{T^{\alpha_2}} c_2 \lambda_{min}(H^{a}) \| z^{a}\|^{2-\alpha_2}  \\
&-k c_1 \lambda_{min}(H^{a})^{\alpha_{1}}  \| z^{a}-g^{a}\|^{2-\alpha_1}- k c_2 \lambda_{max}(H^{a})^{\alpha_{2}} \| z^{a}-g^{a}\|^{2-\alpha_2} \\
&+2 k c_1 \dfrac{p_1}{2-\alpha_1}  \lambda_{max}(H^{a})^{2-\alpha_{1}} \|z^{a}\|^{2-\alpha_1}  + \dfrac{2 kc_1(1-\alpha_1)}{2-\alpha_1} p_1^{-1/(1-\alpha_{1})} \lambda_{min}(H^{a})^{\alpha_1-2} \|z^{a}-g^{a}\|^{2-\alpha_1} \\
 &+2k c_2 \dfrac{p_2}{2-\alpha_2}  \lambda_{max}(H^{a})^{2-\alpha_2} \|z^{a}\|^{2-\alpha_2}  + \dfrac{2k c_2(1-\alpha_2)}{2-\alpha_2} p_2^{-1/(1-\alpha_{2})} \lambda_{min}(H^{a})^{\alpha_2-2} \|z^{a}-g^{a}\|^{2-\alpha_2} \\ 
 &+ k \|z^{a}+g^a\| \|\tilde{\xi}^{a}\| \gamma(H^{a^{-1}}(z^{a}-g^{a})) \lambda_{max}(H^{a^{-1}})\|z^{a}-g^{a}\| \\
&+ \|z^{a} \| \|H^{a^{-1}}(z^{a}-g^{a}) \| \phi(\|\tilde{\xi}^{a} \|).
\end{align*}
\endgroup
We must bound the second last term of the last inequality. We first note that: $\|z^a+g^a\|\leq 2\|z^a\| + \|z^a-g^a\|$. This yields:
\begin{equation} \label{eq:ineqc}
\begin{aligned}
 k \|z^{a}+g^a\| \|\tilde{\xi}^{a}\| \gamma(H^{a^{-1}}(z^{a}-g^{a})) \lambda_{max}(H^{a^{-1}})\|z^{a}-g^{a}\| & \leq 2 k \|z^a\|  \|\tilde{\xi}^{a}\| \gamma(H^{a^{-1}}(z^{a}-g^{a})) \lambda_{max}(H^{a^{-1}})\|z^{a}-g^{a}\| \\
&+k \|\tilde{\xi}^{a}\| \gamma(H^{a^{-1}}(z^{a}-g^{a})) \lambda_{max}(H^{a^{-1}})\|z^{a}-g^{a}\|^2.
\end{aligned}
\end{equation}
Using the approach of Lemma \ref{lem1}, we apply Young's inequality to bound the first term on the right hand side of \eqref{eq:ineqc} as follows:
\begin{align*}
2 k \|z^a\|  \|\tilde{\xi}^{a}\| \gamma(H^{a^{-1}}(z^{a}-g^{a})) \lambda_{max}(H^{a^{-1}})\|z^{a}-g^{a}\| \leq & \|\tilde{\xi}^{a}\| \bigg( 2 k c_1 \dfrac{m_1}{2-\alpha_1}  \lambda_{max}(H^{a})^{2-\alpha_{1}} \|z^{a}\|^{2-\alpha_1} \\
 & \hspace{0.05in} + \dfrac{2 kc_1(1-\alpha_1)}{2-\alpha_1} m_1^{-1/(1-\alpha_{1})} \lambda_{min}(H^{a})^{\alpha_1-2} \|z^{a}-g^{a}\|^{2-\alpha_1} \\
 &+2k c_2 \dfrac{m_2}{2-\alpha_2}  \lambda_{max}(H^{a})^{2-\alpha_2} \|z^{a}\|^{2-\alpha_2} \\
 & \hspace{0.05in} + \dfrac{2k c_2(1-\alpha_2)}{2-\alpha_2} m_2^{-1/(1-\alpha_{2})} \lambda_{min}(H^{a})^{\alpha_2-2} \|z^{a}-g^{a}\|^{2-\alpha_2}  \bigg)
 \end{align*}
 for some positive constants $m_1>0$ and $m_2>0$. As a result, one can find a constant $M_1$ such that:
 \begin{align*}
&2 k \|z^a\|  \|\tilde{\xi}^{a}\| \gamma(H^{a^{-1}}(z^{a}-g^{a})) \lambda_{max}(H^{a^{-1}})\|z^{a}-g^{a}\| \leq  \|\tilde{\xi}^{a}\| M_1 \bigg(  \|z^{a}\|^{2} + \|z^{a}-g^{a}\|^2 \bigg).
 \end{align*}
Next we bound the second term of \eqref{eq:ineqc}. Using the development above, we can write:
\begin{align*}
k \|\tilde{\xi}^{a}\| \gamma(H^{a^{-1}}(z^{a}-g^{a})) \lambda_{max}(H^{a^{-1}})\|z^{a}-g^{a}\|^2 \leq k& \|\tilde{\xi}^{a}\|  \lambda_{max}(H^{a^{-1}}) \bigg( c_1  \lambda_{max}(H^{a})^{\alpha_{1}} \| z^{a}-g^{a}\|^{2-\alpha_{1}} \\
&+ c_2  \lambda_{min}(H^{a})^{\alpha_{2}} \| z^{a}-g^{a}\|^{2-\alpha_{2}}\bigg).
\end{align*}
It follows that there exists a constant $M_2$ such that:
\begin{align*}
k \|\tilde{\xi}^{a}\| \gamma(H^{a^{-1}}(z^{a}-g^{a})) \lambda_{max}(H^{a^{-1}})\|z^{a}-g^{a}\|^2 \leq  \|\tilde{\xi}^{a}\| M_2 \| z^{a}-g^{a}\|^{2}.
\end{align*}
Finally, we consider the indeterminate term: $\|z^{a} \| \|H^{a^{-1}}(z^{a}-g^{a}) \| \phi(\|\tilde{\xi}^{a} \|)$. It can be bounded as:
\begin{align*}
\|z^{a} \| \|H^{a^{-1}}(z^{a}-g^{a})  &\| \phi(\|\tilde{\xi}^{a} \|) \leq \frac{ \phi(\|\tilde{\xi}^{a} \|)}{2 \lambda_{min}(H^{a})} ( \|z^a\|^2+\|z^a-g^a\|^2).
\end{align*}

Next, we pick $p_1$ and $p_2$ such that
 $$
 p_1 > 2\bigg( \dfrac{1}{\lambda_{min}(H^{a})^2} \dfrac{1-\alpha_1}{2-\alpha_1} \bigg)^{1-\alpha_{1}},
 $$
 and 
 $$
 p_2> 2\bigg( \dfrac{1}{\lambda_{min}(H^{a})^2} \dfrac{1-\alpha_2}{2-\alpha_2} \bigg)^{1-\alpha_{2}}.
$$
If one picks $K$ to be sufficiently large
\begin{align*}
K> 2 k  \max\bigg[&\dfrac{p_1T (\lambda_{max}(H^{a})^{2-\alpha_1}}{\lambda_{min}(H^{a})},  \dfrac{p_2T^{\alpha_2} (\lambda_{max}(H^{a})^{2-\alpha_2}}{\lambda_{max}(H^{a})}\bigg]
\end{align*}
then there exist positive constants $p_3$, $p_4$, $p_5$ and $p_6$ such that:
\begin{align*}
\dot{V}_3^{a} \leq &-p_3 \|z^{a}\|^{2-\alpha_1} -p_5 \| z^{a}\|^{2-\alpha_2}  -p_4  \| z^{a}-g^{a}\|^{2-\alpha_1}-  p_6 \| z^{a}-g^{a}\|^{2-\alpha_2} \\
&+ \|\tilde{\xi}^{a}\| M_1 \bigg(  \|z^{a}\|^{2} + \|z^{a}-g^{a}\|^2 \bigg) + \|\tilde{\xi}^{a}\| M_2 \| z^{a}-g^{a}\|^{2}  + \frac{ \phi(\|\tilde{\xi}^{a} \|)}{2\lambda_{min}(H^{a})} ( \|z^a\|^2+\|z^a-g^a\|^2).
\end{align*}
We can then let $M_3=M_1+M_2$ and write:
\begin{align*}
\dot{V}_3^{a} \leq &-p_3 \|z^{a}\|^{2-\alpha_1} -p_5 \| z^{a}\|^{2-\alpha_2} -p_4  \| z^{a}-g^{a}\|^{2-\alpha_1}-  p_6 \| z^{a}-g^{a}\|^{2-\alpha_2} \\
&+\bigg(  \|\tilde{\xi}^{a}\| M_3+ \frac{ \phi(\|\tilde{\xi}^{a} \|)}{2\lambda_{min}(H^{a})}\bigg) (  \|z^{a}\|^{2} + \|z^{a}-g^{a}\|^2 ).
\end{align*}

It is then easy to show that there exists constants $\bar{p}_3$, $\bar{p}_4$ such that:
\begin{align*}
\dot{V}_3^{a} &\leq -\bar{p}_3 \|z^{a}\|^{2-\alpha_1} -\bar{p}_4  \| z^{a}-g^{a}\|^{2-\alpha_1}
\end{align*}
for any $\|\tilde{\xi}^{a}\|$ such that: 
\begin{equation}\label{eq:ineqf}
\begin{aligned} 
\|z^{a}\|^{-\alpha_2}+&\|z^{a}-g^{a}\|^{-\alpha_2}  \geq \frac{\frac{\phi(\|\tilde{\xi}^{a}\|)}{2\lambda_{min}(H^{a})}+M_3 \|\tilde{\xi}^a\|}{\min([p_5,p_6])}.
\end{aligned}
\end{equation}

We can proceed as before to show that there exists a positive constant $p_7$ such that:
\begin{align*}
\dot{V}_3^{a} &\leq -2^{(2-\alpha_1)/2}p_7 V_3^{a^{(2-\alpha_1)/2}}
\end{align*}
subject to \eqref{eq:ineqf}. Thus, the system is the interconnection of an FTISS system and a globally finite-time stable system. Following \cite{hong2010FTISS}, we conclude that the resulting system has a finite-time stable equilibrium at $z^{a}=0$, $g^{a}=0$ and $\tilde{\xi}^{a}=0$. This completes the proof.
\end{proof}

The main result of the previous analysis is that the averaged system (\ref{eq:crazav}) recovers the performance of the proposed target Newton seeking system (\ref{eq:tarsys2}). It remains to prove that the trajectories of the Newton seeking system (\ref{eq:ftNS}) are close to the trajectories of the target system. The main challenge in this analysis is that the dynamics of the Newton seeking system are not Lipschitz. As a result, standard averaging results cannot be applied. 

As in the analysis of the finite-time ESC reported in \cite{guayFT2}, we apply a classical averaging theorem due to Krasnosel'skii-Krein \cite{krasnosel1955principle}. This theorem can be used to show the closeness of solution of the nominal system and the averaged system over a compact set $D\subset \mathbb{R}^2$ as $a \rightarrow 0$.  

The theorem can be stated as follows.

\begin{theorem}\label{theo:KKT}\cite{krasnosel1955principle}
Consider the nonlinear system $\dot{X}=f(t,X,\epsilon)$ where,
\begin{enumerate}
\item the map $f(t,X,\epsilon)$ is continuous in $t$ and $X$ on $\mathbb{R}_{\geq 0} \times \mathbb{R}^n$,
\item there exists a positive constant $L>0$ and a compact set $D \subset \mathbb{R}^n$ such that $\|f(t,X,\epsilon)\| \leq L$ for $t\in \mathbb{R}_{\geq 0}$, $X\in D$ and $\epsilon \in [0,\epsilon^\ast]$,
\item the averaged system 
\begin{align*}
\dot{X}^{a}= \lim_{T \rightarrow \infty} \frac{1}{T} \int_0^T f(t,X^{a},0) dt,\;X^{a}(0)=X(0),
\end{align*}
exists with solutions defined on the set $D$. 
\end{enumerate}
Then, for any $\epsilon\leq \epsilon^{\ast}$, there exists constants $\delta$ and $\overline{T}$ such that:
\begin{align*}
\| X(t)- X^{a}\| \leq \delta
\end{align*}
for $t \in [0,\overline{T}]$.
\end{theorem}

We can now state the final result of this study. 

\begin{theorem} \label{theo:final}
Consider the finite-time Newton seeking system \eqref{eq:ftNS}. Let Assumption \ref{assum:cost1} be satisfied. Then there exists an $a^\ast$ such that for all $a \in (0,a^\ast]$, the optimum $x=x^\ast$, $v=0$, $\xi=vec(H^\ast)$ is a semi-globally practically finite-time stable equilibrium of system \eqref{eq:ftNS}.
\end{theorem}
\begin{proof}
The proof proceeds in two steps. In the first step, we consider the application of Theorem \ref{theo:KKT}. For the analysis of the proposed finite-time Newton seeking system, the  Krasnosel'skii-Krein theorem can be applied as follows. 

Consider the state, $X=[x,g,\xi]^\top $, and the corresponding averaged variables $X^{a}=[x^{a},g^{a},\xi^{a}]^\top $. By the analysis provided above, the averaged system has a finite-time stable equilibrium at the origin $X^\ast=[x^\ast,0,vec(H^\ast)]^\top $. Furthermore, the solutions of the system \eqref{eq:crazav} exist and can be contained in a compact set $D \in \mathbb{R}^{2p+p(p+1)/2}$ containing $X^\ast$. Consider the nonlinear system \eqref{eq:ftNS}. By the smoothness of the cost function $h(x)$ and the periodicity of the dither signal, it follows that the right hand side of the system can be bounded on the compact set $D\in \mathbb{R}^{2p+p(p+1)/2}$ uniformly in $t$.  The continuity and the boundedness of the right hand side of \eqref{eq:ftNS} over a compact set $D$ guarantees existence of solution of the averaged system. As a result, one can invoke the Krasnosel'skii-Krein theorem to guarantee that for any $a \in (0,a^\ast)$ there exists a $\overline{T}$ and a $\delta$ such that:
\begin{align*}
\| X(t)- X^{a}(t)\| \leq \delta
\end{align*}
for $t\in [0,\overline{T}]$.

In the second step, we exploit the finite-time stability of the averaged system and the averaging result established in the first step to establish the finite-time practical semi-global stability of the ESC system. 

Using the finite-time stability property of the averaged system and the averaging result for small amplitude signals, one can apply the approach in the proof of Theorem 1 in \cite{1231248} to show that there exist a generalized class $\mathcal{K}_\infty$ function, $\beta_X$ and a constant, $c_X$, such that:
\begin{align*}
\| X(t) \| \leq \beta_X(\|X(t_0)\|, t) + c_X
\end{align*}
for $X(t_0)\in D$.  This completes the proof.
\end{proof}

\section{Simulation Study}\label{sec4}

We consider the minimization of the cost function: $y=1+30(x_1-1)^2+25(x_1-1)(x_2-2)+15(x_2-2)^2$. The finite extremum seeking controller is implemented with the following tuning parameters:  $a=1$, $\omega_1=150$, $\omega_2=200$, $q_1=3$, $q_2=1.5$, $c_1=1$, $c_2=10^{-4}$, $k=5$, $K=10$ and $K_2=100$. The initial conditions are: $x(0)=[0,1]^\top $, $v(0)=0.01$ and $\xi(0)=[1,0,1]^\top $. The simulation results are shown in Figure \ref{fig1}.  Figure \ref{fig1} shows the trajectories of the decision variables $\hat{x}_1$ and $\hat{x}_2$ along with the trajectories of the target average system \eqref{eq:tarsys}. It also shows the corresponding value of the cost function for the Newton seeking system and the average system. The results demonstrate that the Newton seeking controller approximates the trajectories of the target finite-time system. 
%
\begin{figure}[htbp!]
\centering   
\includegraphics[width=0.7\linewidth]{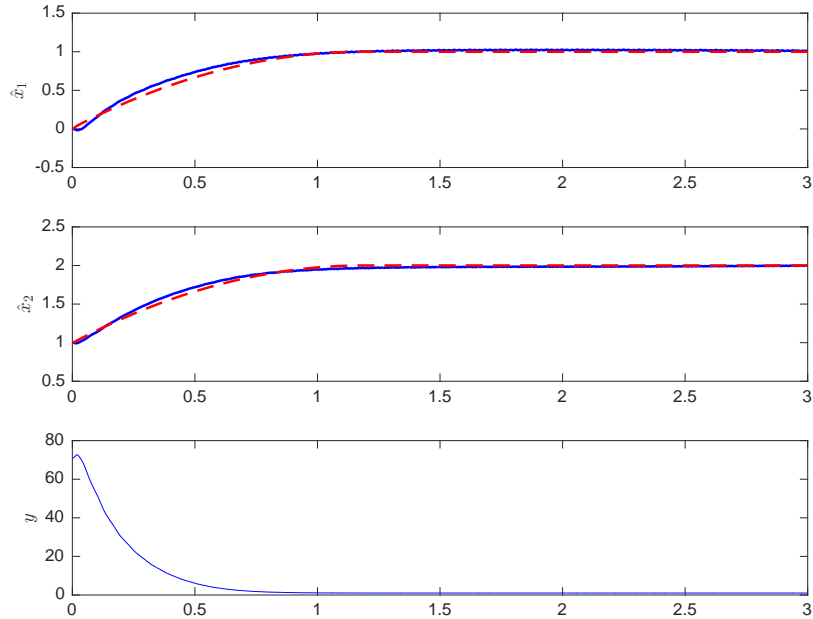}
\caption{Performance of the Finite-time Newton seeking system. The graph shows the decision variables ($\hat{x}_1$, $\hat{x}_2$) (solid line) and the target average variables ($x_1^{a}$, $x_2^{a}$) (dashed line) and the cost function $y$.}
\label{fig1}
\end{figure}

To highlight the finite-time property of the system, we consider the performance of the Newton seeking system with varying initial conditions. In Figure \ref{fig2}, we show the effect of changing the initial conditions on $x_1(0)$ from 0 to 2 with $x_2(0)=1$. We repeat the exercise by changing $x_2(0)$ from 0 to 4 with $x_1(0)=0$. The corresponding trajectories of the system are shown in Figure \ref{fig3}.

\begin{figure}[htbp!]
\centering   
\includegraphics[width=0.7\linewidth]{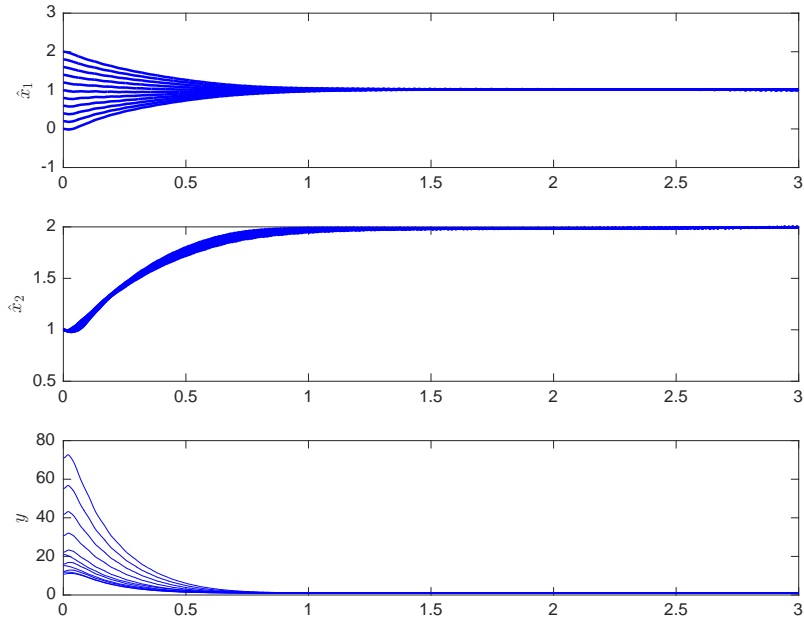}
\caption{Performance of the Finite-time Newton seeking system. The graph shows the decision variables $x_1$ and $x_2$  for the finite-time Newton seeking with varying initial conditions in $x_1$.}
\label{fig2}
\end{figure}

\begin{figure}[htbp!]
\centering   
\includegraphics[width=0.7\linewidth]{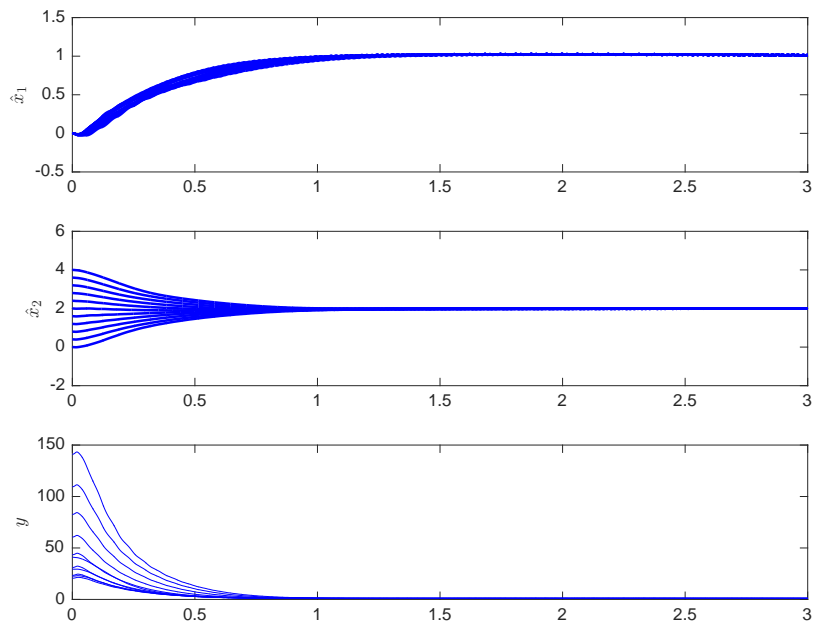}
\caption{Performance of the Finite-time Newton seeking system. The graph shows the decision variables $x_1$ and $x_2$  for the finite-time Newton seeking with varying initial conditions in $x_2$.}
\label{fig3}
\end{figure}

Finally, we consider the effect of changing the gain $k$ on the Newton seeking system. We consider the same tuning parameters and initial conditions. Three gain values are considered , $k=1$, $k=2$ and $k=4$. The results are shown in Figure \ref{fig4}.  The Newton seeking system performs as expected. The increase in gain improves the transient performance of the system. It also reduces the finite convergence time of the system. 

\begin{figure}[htbp!]
\centering   
\includegraphics[width=0.7\linewidth]{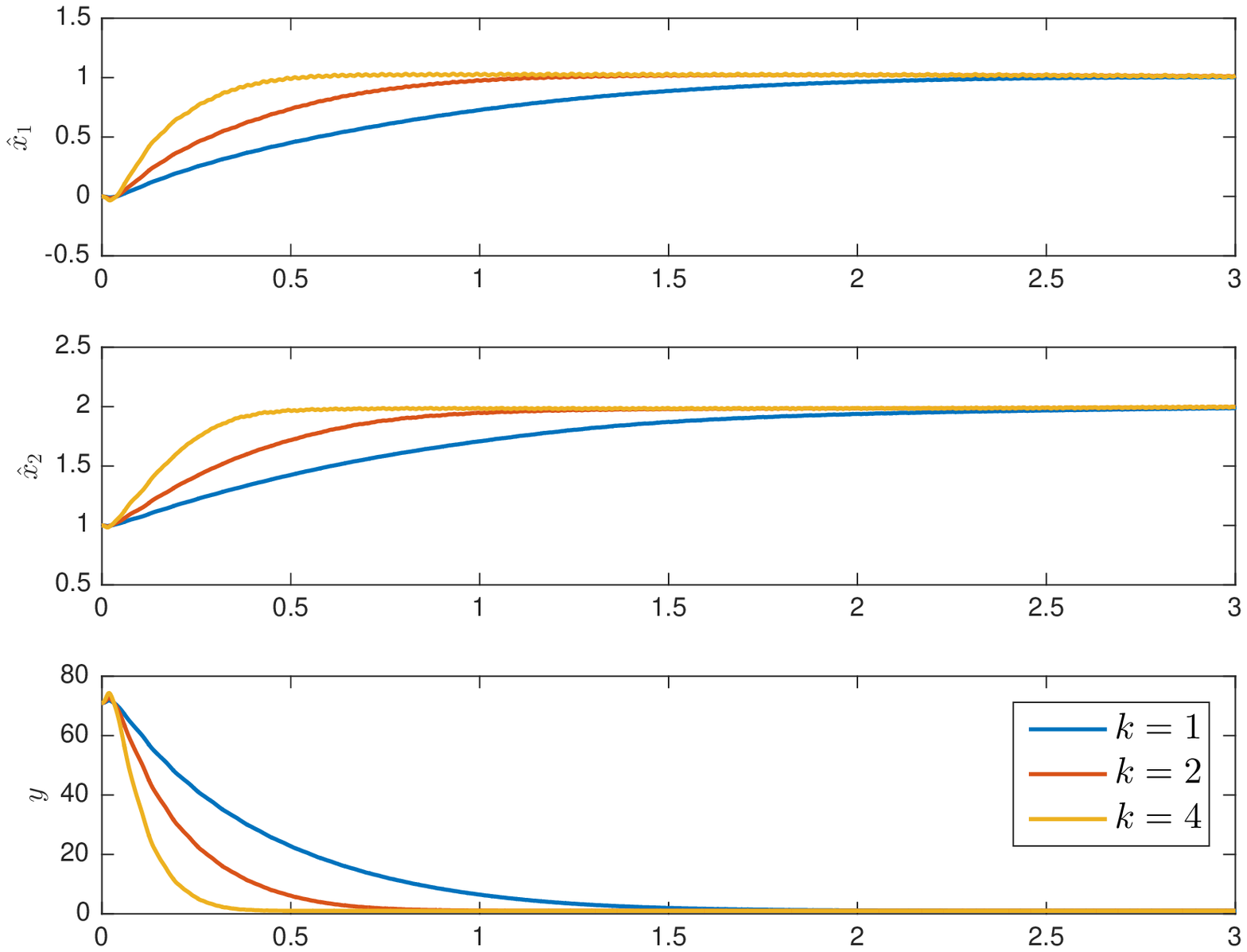}
\caption{Performance of the Finite-time Newton seeking system. The graph shows the decision variables $x_1$ and $x_2$  for the finite-time Newton seeking with varying gains $k$.}
\label{fig4}
\end{figure}

\vspace{-0.5cm}\section{Conclusion} \label{sec5}
In this study, we proposed a Newton seeking design for the solution of real-time optimization problems for unknown multivariate static maps that achieves finite-time convergence to the unknown optimum. The proposed Newton seeking system yields an averaged system with a finite-time stable equilibrium at the unknown optimum. Using classical averaging results for dynamical systems with continuous right hand sides, it is shown that the optimum is a finite-time practically semi-globally stable equilibrium of the ESC system.

\appendix
\begin{proof} (Proof of Lemma \ref{lem2})
As in the proof for the target system, we propose the following change of coordinates $z^{a}=H^\ast v^{a}+\nabla h(x^{a})$ and $g^{a}=\nabla h(x^{a})$ and write the average dynamics as follows:
\begin{equation} \label{eq:crazav2}
\begin{aligned}
\dfrac{\d g^{a}}{\d t} =& \,{k}H^{\ast} \gamma(H^{\ast^{-1}}(z^{a} -g^{a})) H^{\ast^{-1}}(z^{a} { -}g^{a}) \\
\dfrac{\d z^{a}}{\d t}=&\,  -\frac{K}{T}H^\ast \int_0^{T} \bigg( \frac{c_1 (H^\ast v^{a}+ \delta_1(x^{a},t))}{\| H^\ast v^{a}+ \delta_1(x^{a},t)\|^{\alpha_{1}}} + \frac{c_2 (H^\ast v^{a}+ \delta_1(x^{a},t))}{\|H^\ast v^{a}+ {\delta_1}(x^{a},t)\|^{\alpha_{2}}} \bigg) dt  \,\\
&+{k}H^{\ast} \gamma(H^{\ast^{-1}}(z^{a} -g^{a}))(z^{a} -g^{a}).
\end{aligned}
\end{equation}
Under this assumption, we pose the Lyapunov function: $V_1^{a} = \frac{1}{2} \| z^{a}\|^2 + \frac{1}{2} \| g^{a}\|^2$. Proceeding as in the proof of Lemma \ref{lem1}, its time derivative is:
\begin{align}
\dot{V}_1^{a} =& -\frac{K}{T} (z^{a})^\top H^\ast  \int_0^{T} \bigg(  \frac{c_1 (H^{\ast} v^{a} + \delta_1(x^{a},t))}{\|H^{\ast} v^{a}  + \delta_1(x^{a},t))\|^{\alpha_{1}}}+ \frac{c_2 (H^{\ast}v^{a} + \delta_1(x^{a},t))}{\|H^{\ast} v^{a} + \delta_1(x^{a},t))\|^{\alpha_{2}}} \bigg) dt \label{eq:ineq1}\\
&- k  \gamma(H^{\ast^{-1}}(z^{a} -g^{a})) \|z^{a} -g^{a}\|^2 \label{eq:ineq1} + 2k  \gamma(H^{\ast^{-1}}(z^{a} -g^{a}))(z^{a})^\top (z^{a} -g^{a}). \nonumber
\end{align}

Let us consider the first term on the right hand side of the last inequality:
\begin{align*}
 \Phi_1=(z^{a})^\top  H^\ast \frac{1}{T}\int_0^T \bigg(  \frac{ (H^{\ast} v^{a}+ \delta_1(x^{a},t))}{\|H^{\ast} v^{a} + \delta_1(x^{a},t))\|^{\alpha_{1}}}\bigg) dt
 \end{align*}
 
 Let $\theta(t)=H^{\ast} v^{a}+\delta_1(x^{a},t)$.  The function $\Phi_1$ can written explicitly as:
 $$\Phi_1 = \bigg( \frac{1}{T} \int_0^T  \theta(t)dt \bigg)^\top  H^\ast \frac{1}{T} \int_0^T  \frac{ \theta(t)}{\|\theta(t)\|^{\alpha_{1}}}   dt.$$
 First, we note that $\Phi_1$ is nonnegative. We then consider the following expression:
$ \Phi_1^\prime  = \Phi_1 \bigg\|\frac{1}{T}\int_{0}^{T} |\theta(t)| dt \bigg\|^{\alpha_1}.$
This can be rewritten as:
 \begin{align*}
\Phi_1^\prime = \bigg( \frac{1}{T} \int_0^T  \theta(t)dt \bigg)^\top  H^\ast  \frac{1}{T} \int_0^T   \theta(t)\frac{\bigg\| \frac{1}{T} \int_{0}^{T}| \theta(t) |dt \bigg\|^{\alpha_1}}{\|\theta(t)\|^{\alpha_{1}}}   dt. 
\end{align*}

By a simple application of Jensen's inequality for concave functions (see e.g. \cite{hardy52}, Theorem 189), it follows that for any $0<\alpha<1$ that
$$
\bigg\| \frac{1}{T} \int_{0}^{T} |\theta(t)| dt \bigg\|^{\alpha} \geq\frac{1}{T} \int_{0}^{T} \|\theta(t)\|^{\alpha} dt.
$$
As a result, we obtain for $0<\alpha_1<1$:
 \begin{align*}
\Phi_1^\prime \geq \bigg( \frac{1}{T} \int_0^T  \theta(t)dt \bigg)^\top  H^\ast \frac{1}{T^2} \int_0^T   \theta(t)\frac{ \int_{0}^{T}\| \theta(t)\|^{\alpha_1} dt}{\|\theta(t)\|^{\alpha_{1}}}   dt. 
\end{align*}
Next, we note that the following inequality holds:
\begin{align*}
\frac{\| \theta(t)\|}{\| \theta(t)\|^{\alpha_{1}}}  \int_{0}^{T}\| \theta(t)\|^{\alpha_1} dt \geq \| \theta(t)\|
\end{align*}
$\forall t \in [0,T]$. Therefore, we can write:
 \begin{align*}
\Phi_1^\prime \geq \bigg( & \int_0^T  \theta(t)dt \bigg)^\top  H^\ast  \frac{1}{T^2} \int_0^T\frac{\theta(t)}{\| \theta(t)\|} \frac{\|\theta(t)\|}{\| \theta(t)\|^{\alpha_1} }\bigg(\int_{0}^{T}\| \theta(t)\|^{\alpha_1} dt \bigg) dt
\end{align*}
and,
\begin{align*}
\Phi_1^\prime=\Phi_1 \bigg\| \frac{1}{T} \int_{0}^{T} |\theta(t)| dt \bigg\|^{\alpha_1} &\geq \frac{\lambda_{min}(H^\ast) }{T}\bigg\| \frac{1}{T}\int_0^T  \theta(t) dt \bigg\|^2 .
\end{align*}
Finally, we obtain:
\begin{align*}
\Phi_1\geq \frac{\lambda_{min}(H^\ast)}{T} \frac{\bigg\| \frac{1}{T}\int_0^T  \theta(t) dt \bigg\|^2}{ \bigg\| \frac{1}{T} \int_{0}^{T} | \theta(t)| dt \bigg\|^{\alpha_1}}=\frac{\lambda_{min}(H^\ast) \|z^{a}\|^2}{T \| z^{a}\|^{\alpha_1}}.
\end{align*}

In a similar fashion, we can define the following function: 
  \begin{align*}
 \Phi_2= \bigg(\frac{1}{T} \int_0^T  &\theta(t) dt \bigg)^\top H^\ast \frac{1}{T} \int_0^T  \frac{ \theta(t)}{\|\theta(t)\|^{\alpha_{2}}}   dt.
 \end{align*}
 We can rewrite $\Phi_2$ to generate a lower bound as follows:
  \begin{align*}
& \Phi_2 \leq\lambda_{min}(H^\ast)  \bigg( \frac{1}{T} \int_0^T  \theta(t) dt \bigg)^\top  \frac{1}{T} \int_0^T  \frac{ \theta(t)}{\|\theta(t)\|^{\alpha_{2}}}   dt 
 + \frac{\lambda_{min}(H^\ast) \left\|\frac{1}{T} \int_0^T  \theta(t) dt \right\|^2}{\left\| \int_0^T |\theta(t)| dt \right\|^{\alpha_2}} 
 - \frac{\lambda_{min}(H^\ast) \left\|\frac{1}{T} \int_0^T  \theta(t) dt \right\|^2}{\left\| \int_0^T |\theta(t)| dt \right\|^{\alpha_2}}.
 \end{align*}
This can be rewritten as:
\begin{align*}
& \Phi_2 \leq\lambda_{min}(H^\ast)  \bigg( \frac{1}{T} \int_0^T  \theta(t) dt \bigg)^\top  \frac{1}{T} \int_0^T  \frac{ \theta(t)}{\|\theta(t)\|^{\alpha_{2}}}   dt  + \frac{\lambda_{min}(H^\ast) \left\|\frac{1}{T} \int_0^T  \theta(t) dt \right\|^2}{\left\| \int_0^T |\theta(t)| dt \right\|^{\alpha_2}} 
 - \frac{\lambda_{min}(H^\ast) \left\|\frac{1}{T} \int_0^T  \theta(t) dt \right\|^2}{\int_0^T \|\theta(t)\|^{\alpha_2} dt}.
 \end{align*}
We see that $\Phi_2 \geq \frac{\|z^a\|^2}{T^{\alpha_{2}}} \|z^a\|^{\alpha_2}$, if $\|\theta(t)\|^\alpha_2 \leq  \int_0^T \|\theta(t)\|^{\alpha_2} dt$. Since  this applies for all $t$, the inequality for $\Phi_2$ holds. One can then substitute the inequalities in \eqref{eq:ineq1} to obtain the following:
\begin{equation*} 
\begin{aligned}
\dot{V}_1^{a} \leq-K \lambda_{min}(H^\ast) \bigg(\frac{c_1}{T} \frac{\|z^{a}\|^2}{\|z^{a}\|^{\alpha_{1}}}+&\frac{c_2}{T^{\alpha_{2}}}  \frac{\|z^{a}\|^2}{\|z^{a}\|^{\alpha_{2}}}\bigg)
- k  \gamma(H^{\ast^{-1}}(z^{a} -g^{a})) \|z^{a} -g^{a}\|^2 \\&+ k  \gamma(H^{\ast^{-1}}(z^{a} -g^{a}))(z^{a})^\top (z^{a} -g^{a}).
\end{aligned}
\end{equation*}
Again, following the arguments in the proof of Lemma \ref{lem1}, there exists a constant $k_8>0$ such that:
\begin{align*}
\dot{V}_1^{a} \leq &-k_8 (2 V_1^{a})^{(2-\alpha_1)/2}=-2^{(2-\alpha_1)/2} k_8 V_1^{a^{(2-\alpha_1)/2}}.
\end{align*}

Since $(2-\alpha_1)/2<1$, it follows that the system has a globally finite-time equilibrium at $z^{a}=0$ and $g^{a}=0$ when $H^{a}=H^\ast$. This completes the proof.
\end{proof}


\clearpage

\begin{thebibliography}{10}
\providecommand \doibase [0]{http://dx.doi.org/}%

\bibitem{Cortes2006}
Cort{\'e}s J. Finite-time convergent gradient flows with applications to
  network consensus. {\it Automatica} 2006\string; 42(11)\string: 1993--2000.

\bibitem{ICML20}
Romero O, Benosman M. Finite-Time Convergence in Continuous-Time Optimization.
  {\it Proceedings of the 37 th International Conference on Machine Learning
  (ICML)} 2020\string: 1--10.

\bibitem{garg2020}
Garg K, Panagou D. Fixed-Time Stable Gradient Flows: Applications to
  Continuous-Time Optimization. {\it IEEE Transactions on Automatic Control,
  early access} 2020\string: 1--14.

\bibitem{IJC20}
Romero O, Benosman M. Time-varying continuous-time optimization with
  pre-defined finite-time stability. {\it International Journal of Control}
  2020\string: 1--30.

\bibitem{5572972}
Tan Y, Moase W, Manzie C, Nesic~and D, Mareels I. Extremum seeking from 1922 to
  2010. {\it Proceedings of the IEEE 29th Chinese Control Conference (CCC)}
  2010\string: 14 -26.

\bibitem{KW:SESFGDS}
Krstic M, Wang H. Stability of Extremum Seeking Feedback for General Dynamic
  Systems. {\it Automatica} 2000\string; 36(4)\string: 595-601.

\bibitem{Tan2006889}
Tan Y, Nesic D, Mareels I. On non-local stability properties of extremum
  seeking control. {\it Automatica} 2006\string; 42(6)\string: 889 - 903.

\bibitem{ghaffari2012multivariable}
Ghaffari A, Krsti{\'c} M, Ne{\v{s}}i{\'c} D. Multivariable Newton-based
  extremum seeking. {\it Automatica} 2012\string; 48(8)\string: 1759--1767.

\bibitem{labar2019newton}
Labar C, Garone E, Kinnaert M, Ebenbauer C. Newton-based extremum seeking: A
  second-order Lie bracket approximation approach. {\it Automatica}
  2019\string; 105\string: 356--367.

\bibitem{gavurin1958nonlinear}
Gavurin MK. Nonlinear functional equations and continuous analogues of
  iteration methods. {\it Izvestiya Vysshikh Uchebnykh Zavedenii. Matematika}
  1958(5)\string: 18--31.

\bibitem{airapetyan1999continuous}
Airapetyan R. Continuous Newton method and its modification. {\it Applicable
  Analysis} 1999\string; 73(3-4)\string: 463--484.

\bibitem{massicot2019line}
Massicot O, Marecek J. On-line non-convex constrained optimization. {\it arXiv
  preprint arXiv:1909.07492} 2019.

\bibitem{garg2018new2}
Garg K, Panagou D. New results on finite-time stability: Geometric conditions
  and finite-time controllers. {\it Proceedings of IEEE American Control
  Conference (ACC)} 2018\string: 442--447.

\bibitem{guayFT2}
Guay M, Benosman M. Finite-time extremum seeking control for a class of unknown
  static maps. {\it International Journal of Adaptive Control and Signal
  Processing} 2020\string: 1-14.

\bibitem{poveda2020fixed}
Poveda JI, Krsti{\'c} M. Fixed-Time Gradient-Based Extremum Seeking. {\it
  Proceedings of IEEE American Control Conference} 2020\string: 2838--2843.

\bibitem{guayNFT}
Guay M. Finite-time Newton seeking control for a class of unknown static maps.
  {\it Proceedings of IEEE Conference on Decision and Control} 2020.

\bibitem{hong2010FTISS}
Hong Y, Jiang ZP, Feng G. Finite-Time Input-to-State Stability and Applications
  to Finite-Time Control Design. {\it SIAM Journal on Control and Optimization}
  2010\string; 48(7)\string: 4395-4418.

\bibitem{krasnosel1955principle}
Krasnosel'skii MA, Krein SG. On the principle of averaging in nonlinear
  mechanics. {\it Uspekhi Matematicheskikh Nauk} 1955\string; 10(3)\string:
  147--152.

\bibitem{P12}
Andrey P. Nonlinear Feedback Design for Fixed-Time Stabilization of Linear
  Control Systems. {\it IEEE Transactions on Automatic Control} 2012\string;
  57(8).

\bibitem{REPP18}
Lopez-Ramirez F, Efimov D, Polyakov A, Perruquetti W. On Necessary and
  Sufficient Conditions for Fixed-Time Stability of Continuous Autonomous
  Systems. {\it 27th European Control Conference (ECC)} 2018.

\bibitem{1231248}
{Teel} AR, {Moreau} L, {Nesic} D. A unified framework for input-to-state
  stability in systems with two time scales. {\it IEEE Transactions on
  Automatic Control} 2003\string; 48(9)\string: 1526-1544.

\bibitem{hardy52}
Hardy G, Littlewood J, Polya G. {\it Inequalities. 2nd Edition}.
\newblock Cambridge: Cambridge University Press .
\newblock 1952.

\end{thebibliography}
	\end{document}